\documentclass[12pt]{article}
\usepackage{amsmath,amssymb,pst-all,graphicx}
\usepackage{amscd, amsthm}
\usepackage[cmtip,arrow]{xy}
\usepackage{pb-diagram, pb-xy}

\newcommand{\CC}{\mathbb C}

\newcommand{\PP}{\mathbb P}

\newcommand{\ZZ}{\mathbb Z}

\newcommand{\mcA}{\mathcal{A}}
\newcommand{\mcB}{\mathcal B}
\newcommand{\mcC}{\mathcal C}

\newcommand{\mcO}{\mathop {\mathcal O}\nolimits}

\newcommand{\mcT}{\mathcal T}

\DeclareMathOperator{\HH}{H}

\newcommand{\Ima}{\mathop {\rm Im}\nolimits}

\newcommand{\id}{\mathop {{\rm id}}\nolimits}

\newtheorem{thm}{Theorem}[section] 
\newtheorem{lem}[thm]{Lemma}     
\newtheorem{cor}[thm]{Corollary}
\newtheorem{prop}[thm]{Proposition}

\newtheorem{claim}{Claim}
\theoremstyle{definition}
\newtheorem{defin}[thm]{Definition}
\newtheorem{rem}[thm]{Remark}

\begin{document}

\title{Connected numbers and the embedded topology of plane curves}
\author{Taketo Shirane}
\date{
}
\maketitle

\begin{abstract}
The splitting number of a plane irreducible curve for a Galois cover is effective to distinguish the embedded topologies of plane curves. 
In this paper, we define a \textit{connected number} of any plane curve for a Galois cover whose branch divisor has no common components with the plane curve, which is similar to the splitting number. 
We classify  the embedded topology of Artal arrangements of degree $b\geq 4$ by the connected number, where an Artal arrangement of degree $b$ is a plane curve consisting of one smooth curve of dgree $b$ and three total inflectional tangents. 
\end{abstract}

\section{Introduction}

In this paper, we study the embedded topologies of plane curves based on the ideas used in \cite{bannai} and \cite{shirane}. 
Here a plane curve $C$ is a reduced divisor on the complex projective plane $\PP^2:=\PP^2_{\CC}$, and the \textit{embedded topology} of the plane curve $C$ is the homeomorphism class of the pair $(\PP^2,C)$. 
We call the homeomorphism class of the pair $(T(C), C)$ the \textit{combinatorics} of a plane curve $C$, where $T(C)$ is a tubular neighborhood $T(C)$ of $C$. 
It is obvious that the combinatorics of a plane curve is an invariant of the embedded topology. 
Moreover, it is known that the combinatorics of two plane curves $C_1, C_2$ are same if and only if $C_1$ and $C_2$ are equisingular (see \cite[Remark~3]{act} for detail). 
Hence the combinatorics gives a simple classification of the embedded topologies. 
However, 
a result of Zariski in \cite{zariski} showed that the classification by the combinatorics is coarse, i.e., the combinatorics does not determine the embedded topologies. 

The aim of this paper is to give a finer classification of the embedded topology of plane curves than the one by the combinatorics. 
Our strategy for the aim above consists of the following two steps (see \cite[pp.5]{act} for other strategies); 
\begin{enumerate}
\item choosing an invariant of the embedded topologies of plane curves, and 
\item proving the existence of a plane curve with a given combinatorial type and an expected value of the invariant. 
\end{enumerate}
As the consequence of the step (ii), we obtain a $k$-plet $(C_1,\dots,C_k)$ of $k$ plane curves for some integer $k\geq 2$ such that, for any $i\ne j$, the combinatorics of $C_i$ and $C_j$ are same, but their embedded topologies are different. 
We call such a $k$-plet a \textit{Zariski $k$-plet}, and \textit{Zariski pair} if $k=2$. 
The step (ii) is equivalent to discover a Zariski $k$-plet by the chosen invariant. 

For the step (i) above, the classical invariant is the fundamental group of the complement of a plane curve.  
Many Zariski pairs discovered by using the fundamental groups as Zariski \cite{zariski} (cf. \cite{act, oka, rybnikov, tokunaga-dihedral}). 
Some invariants different from the fundamental group are chosen in the step (i) above: 
Artal, Cogolludo and Carmona proved that the braid monodromy is a topological invariant of certain plane curves in \cite{acc}, and there is a complement-equivalent Zariski pair whose embedded topologies are distinguished by the braid monodromy (\cite{act}). 
In \cite{benoit-meilhan}, Guerville and Meilhan defined an invariant of the embedded topology of plane curves, called the \textit{linking set}; and they discovered a Zariski pair of line arrangements. 
In this paper, we give a new topological invariant, called a \textit{connected number}. 
This invariant is based on the following studies about \textit{splitting curves}: 

Artal and Tokunaga studied splitting curves for double covers to prove difference of the fundamental group in \cite{artal-tokunaga}. 
Bannai defined the \textit{splitting type} of splitting curves, and gives a method to prove the difference of embedded topologies of plane curves not through the fundamental group in \cite{bannai}. 
In \cite{shirane}, the author defined the \textit{splitting number} of irreducible curves for Galois covers based on Bannai's idea; and by using the splitting number, he proved that the $\pi_1$-equivalent equisingular curves defined by Shimada in \cite{shimada} provide $\pi_1$-equivalent Zariski $k$-plets. 
This result shows importance of study of splitting curves to distinguish embedded topology of plane curves. 

In this paper, we define the connected number of plane curves (not necessarily irreducible) for Galois covers, which is an invariant under certain homeomorphisms from $\PP^2$ to itself (see Definition~\ref{def. connected number} and Proposition~\ref{prop. inv}). 
The connected number is similar to the splitting number, but not its generalization (see Remark~\ref{rem. splitting and connected}). 
We give a method to compute the connected number of certain plane curves for cyclic covers (see Theorem~\ref{thm. cn} and Corollary~\ref{cor. cn}). 
Finally, we classify the embedded topologies of Artal arrangements of degree $b\geq 4$ by the connected number, i.e., we give Zariski $k$-plets of Artal arrangements of $b\geq 4$, where an Artal arrangement of degree $b$ is a plane curve consisting of one smooth curve of degree $b$ and three lines (the Artal arrangement of degree $3$ was first studied by Artal \cite{artal}, see Section~4 for detai).

\section{Connected number}\label{sec. connected number}

In this section, we define the connected number of plane curves for Galois covers (Definition~\ref{def. connected number}), and state the main results about the connected number for cyclic covers (Theorem~\ref{thm. cn} and Corollary~\ref{cor. cn}).

\begin{defin}\label{def. connected number}
Let $Y$ be a smooth variety, 
and let $\phi:X\to Y$ be a Galois cover branched along $B$. 
Let $C\subset Y$ be an algebraic subset of $Y$ such that all irreducible components of $C$ are not contained in $B$, and $C\setminus B$ is connected. 
We call the number of connected components of $\phi^{-1}(C\setminus B)$ the \textit{connected number} of $C$ for $\phi$, and denote it by $c_{\phi}(C)$. 
\end{defin}

\begin{rem}\label{rem. splitting and connected}
\begin{enumerate}
\item
It is obvious that the connected number $c_{\phi}(C)$ divides the degree of $\phi$. 
\item
For an irreducible plane curve $C\subset\PP^2$, we have 
$c_{\phi}(C)\leq s_{\phi}(C)$, 
where $s_{\phi}(C)$ is the splitting number of $C$ for $\phi$. 
There exists a nodal irreducible plane curve $C$ with a contact conic $\Delta$ such that $s_{\phi}(C)=2$ and $c_{\phi}(C)=1$ for the double cover $\phi:X\to\PP^2$ branched at $\Delta$ (see \cite{bannai-shirane}). 
If $C$ is non-singular, then $c_{\phi}(C)=s_{\phi}(C)$. 
\end{enumerate}
\end{rem}

By \cite[Proposition~1.3]{shirane}, we obtain the following proposition. 

\begin{prop}\label{prop. inv}
Let $B_i$ ($i=1,2$) be reduced divisors on a smooth variety $Y$, and 
let $G$ be a finite group. 
For each $i=1,2$, let $\phi_i:X_i\to Y$ be the $G$-cover induced by a surjection $\theta_i:\pi_1(Y\setminus B_i)\twoheadrightarrow G$. 
Let $C_1$ be an algebraic subset of $Y$ such that all irreducible components of $C_1$ are not contained in $B_1$, and $C_1\setminus B_1$ is connected. 
Assume that there exists a homeomorphism $h:Y\to Y$ with $h(B_1)=B_2$ and an automorphism $\sigma:G\to G$ such that  $\sigma\circ\theta_2\circ h_{\ast}=\theta_1$. 
Then $c_{\phi_1}(C_1)=c_{\phi_2}(C_2)$, where $C_2=h(C_1)$. 
\end{prop}

For a plane curve $B\subset\PP^2$ and a cyclic cover $\phi:X\to\PP^2$ induced by a surjection $\theta:\pi_1(\PP^2\setminus B)\twoheadrightarrow\ZZ/m\ZZ$, 
let $\mcB_{\theta}$ denote the following divisor: 
\[ \mcB_{\theta}:=\sum_{i=1}^{m-1}i\cdot B_i, \]
where $B_i$ is the sum of irreducible components whose meridians are mapped to $[i]\in\ZZ/m\ZZ$ by $\theta$ for each $i=1,\dots,m-1$. 
In this case, the degree of $\mcB_{\theta}$ is divided by $m$ by \cite[Example~2.1]{pardini}, 
and we call $\phi:X\to\PP^2$ the \textit{cyclic  cover of degree $m$ branched at $\mcB_{\theta}$}. 

Conversely, 
for a divisor $\mcB=\sum_{i=1}^{m-1}i\cdot B_i$ such that $B=\sum_{i=1}^{m-1}B_i$ is reduced, 
if the degree of $\mcB$ is divided by $m$, then there exists the cyclic cover of degree $m$ branched at $\mcB$ by \cite[Example~2.1]{pardini}. 
The cyclic cover $\phi_{\mcB}:X_{\mcB}\to\PP^2$ branched at $\mcB$ is constructed as follows: 

Let $n$ be the quotient of $\deg \mcB$ by $m$, $\deg\mcB=mn$. 
For the invertible sheaf $\mcO(n)$ on $\PP^2$, let $p_n:\mcT_n\to\PP^2$ be the line bundle associated to $\mcO(n)$. 
Let $t_n\in\HH^0(\mcT_n,p_n^{\ast}\mcO(n))$ be a tautological section. 
The variety $X_{\mcB}$ is constructed as the normalization of $X'_{\mcB}$ defined by $t_n^m=p_n^{\ast}F_{\mcB}$ in $\mcT_n$, 
where $F_{\mcB}\in\HH^0(\PP^2,\mcO(mn))$ is a global section defining $\mcB$, and $\phi_{\mcB}:X_{\mcB}\to\PP^2$ is the composition of the normalization $X_\mcB\to X_\mcB'$ and $p_n:\mcT_n\to\PP^2$. 

Note that a global section $F\in \HH^0(\PP^2,\mcO(n))$ of $\mcO(n)$ corresponds to a section $\PP^2\to\mcT_{n}$ of the line bundle $p_{n}$, denote it by the same notation $F:\PP^2\to\mcT_{n}$. 
Since we have the morphism $p_{\mathrm{pow}}:\mcT_{n}\to\mcT_n^{\otimes \mu}\cong\mcT_{\mu n}$ defined locally by $(P,t)\mapsto (P,t^\mu)$ for $P\in\PP^2$ and $t\in\CC$, we can regard the equation $t_n^{\mu}-h=0$ as a defining equation of $p_{\mathrm{pow}}^{-1}(\Ima(h))$ for a continuous map $h:\mcC\to\mcT_{\mu n}$ from a reduced divisor $\mcC$ on $\PP^2$ to $\mcT_{\mu n}$ satisfying $p_n\circ h=\id_{\mcC}$. 


\begin{thm}\label{thm. cn}
Let $B=\sum_{i=1}^{m-1}B_i$ and $\mcC$ be two plane curves on $\PP^2$ with no common components. 
Assume that the degree $\mcB=\sum_{i=1}^{m-1}i\cdot B_i$ is divided by $m$, 
that all irreducible components of $\mcC$ are smooth, and 
that $\mcC$ is smooth at all intersection points of $\mcC$ and $B$. 
Let $\phi_{\mcB}:X_{\mcB}\to\PP^2$ be the cyclic cover branched at $\mcB$, and  put $n:=\deg \mcB/m$. 
Let $F_{\mcB}\in \HH^0(\PP^2,\mcO(mn))$ be a global section defining $\mcB$. 
Then, for a divisor $\lambda$ of $m$, say $m=\mu \lambda$, the following conditions are equivalent: 
\begin{enumerate}
\item $c_{\mcB}(\mcC)=\lambda$, where $c_{\mcB}(\mcC):=c_{\phi_{\mcB}}(\mcC)$; and 
\item $\lambda$ is the maximal divisor of $m$ such that there exists a continuous map $h:\mcC\to \mcT_{\mu n}$ satisfying the following conditions; 
\begin{enumerate}
\item\label{thm. cond1} 
	$p_{\mu n}\circ h=\id_{\mcC}$; 
\item\label{thm. cond2} 
	$h^\lambda=h^{\otimes\lambda}:\mcC\to\mcT_{\mu n}^{\otimes\lambda}=\mcT_{mn}$ is coincide with the restriction $F_{\mcB}|_{\mcC}$ of $F_{\mcB}$ to $\mcC$; and
\item\label{thm. cond3} 
	for each irreducible component $C$ of $\mcC$, there exists a global section $g_C\in \HH^0(\PP^2,\mcO(\mu n))$ such that $h|_C=g_C|_C$. 
\end{enumerate}
\end{enumerate}
\end{thm}

\begin{cor}\label{cor. cn}
Assume the same assumption of Theorem~\ref{thm. cn}, and that $\mcC$ is a nodal curve. 
Then, for a divisor $\lambda$ of $m$, the following conditions are equivalent:
\begin{enumerate}
\item $c_{\mcB}(\mcC)=\lambda$; and 
\item $\lambda$ is the maximal divisor of $m$ such that there exists a divisor $D$ on $\PP^2$ with $\lambda D|_{\mcC}=\mcB|_{\mcC}$ as Cartier divisors on $\mcC$.  
\end{enumerate}
\end{cor}

\section{Proofs}

In this section, we prove Theorem~\ref{thm. cn} and Corollary~\ref{cor. cn}. 
We first prove Theorem~\ref{thm. cn}. 

\begin{proof}[Proof of Theorem~\ref{thm. cn}]
Let $B=\sum_{i=1}^{m-1}B_i$ and $\mcC$ be two curves on $\PP^2$ with no common components. 
Assume that all irreducible components of $\mcC$ are smooth, that $\mcC$ is smooth at each intersection of $B$ and $\mcC$, and the degree of $\mcB=\sum_{i=1}^{m-1}i\cdot B_i$ is divided by $m$. 
Let $\phi_\mcB:X_\mcB\to\PP^2$ be the cyclic cover branched at $\mcB$. 
Let $p_n:\mcT_n\to\PP^2$, $F_\mcB$ and $\phi'_\mcB:X'_\mcB\to\PP^2$ be the notation used in the construction of $\phi_\mcB:X_\mcB\to\PP^2$ in Section~\ref{sec. connected number}. 
Let $\kappa_\mcB:X_\mcB\to X'_\mcB$ be the morphism of the normalization. 
Since $X'_\mcB$ is smooth over $\PP^2\setminus B$, the normalization $\kappa_\mcB$ gives an isomorphism $\kappa'_\mcB$ over $\PP^2\setminus B$: 
\[ \kappa'_\mcB:=\kappa|_{X_\mcB\setminus\phi^{-1}(B)}:X_\mcB\setminus\phi_\mcB^{-1}(B)\overset{\sim}{\to}X'_\mcB\setminus(\phi'_\mcB)^{-1}(B) \]
Hence the connected number $c_{\mcB}(\mcC)$ is equal to the number of connected components of $(\phi'_\mcB)^{-1}(\mcC\setminus B)$. 
Let $f_\mcC=0$ denote a defining equation of $\mcC$. 

If there exists a continuous map $h:\mcC\to\mcT_n$ satisfying the conditions (ii-a), (ii-b) and (ii-c), then $t_n^{\mu}-\zeta_{\lambda}^{j-1}h=f_\mcC=0$ defines a subset $\widetilde{\mcC}_{j}$ of $(\phi'_\mcB)^{-1}(\mcC)$ for each $j=1,\dots,\lambda$ such that $\bigcup_{j=1}^{\lambda}\widetilde{\mcC}_j=(\phi'_\mcB)^{-1}(\mcC)$, and 
$\widetilde{\mcC}_j\setminus(\phi'_\mcB)^{-1}(B)$ ($j=1,\dots,\lambda$) are disconnected. 
Hence we have $c_{\mcB}(\mcC)\geq \lambda$. 
Thus it is sufficient to prove that there exists a continuous map $h:\mcC\to\mcT_n$ satisfying the conditions (ii-a), (ii-b) and (ii-c) such that $\widetilde{\mcC}_j\setminus(\phi'_\mcB)^{-1}(B)$ is connected, where $\widetilde{\mcC}_j$ is the subset of $(\phi'_\mcB)^{-1}(\mcC)$ defined by $t_n^{\mu}-\zeta_{\lambda}^{j-1}h=f_\mcC=0$. 

Let $\mcC=\sum_{i=1}^{k}C_i$ be the irreducible decomposition of $\mcC$. 
We prove the existence of such $h:\mcC\to\mcT_n$ by induction on the number $k$ of irreducible components of $\mcC$. 
If $\mcC$ is irreducible, hence $\mcC$ is smooth, then Theorem~\ref{thm. cn} is equivalent to \cite[Theorem~2.1]{benoit-shirane} (cf. \cite[Theorem~2.7]{shirane}). 
Moreover, the closure of each connected component of $\phi_\mcB^{-1}(\mcC\setminus B)$ is defined by $t_n^{m/\lambda}-\zeta_{\lambda}^{j-1}h=f_\mcC=0$ in $\mcT_{n}$ for $j=1,\dots,\lambda$, where $\zeta_{\lambda}$ is a primitive $\lambda$th root of the unity. 

Suppose that $k>1$, and let $\mcC_1:=\sum_{i=1}^{k-1}C_i$ and $\mcC_2:=C_k$ with $c_{\mcB}(\mcC_i)=\nu_i$ for $i=1,2$. 
Moreover, we suppose that the closure $\widetilde{\mcC}_{i j}$ of each connected component $\widetilde{\mcC}'_{i j}$ of $(\phi'_\mcB)^{-1}(\mcC_i\setminus B)$ is defined by $t_n^{\mu_i}-\zeta_{\nu_i}^{j-1}h_i=f_i=0$ in $\mcT_{n}$ for $i=1,2$ and $j=1,\dots,\nu_i$, 
where $\mu_i:=m/\nu_i$, $f_i=0$ is a defining equation of $\mcC_i$, and 
$h_i:\mcC_i\to\mcT_{\mu_i n}$ is a continuous map satisfying the conditions (ii-a), (ii-b) and (ii-c) for $\mcC_i$. 
Let $\nu$ be the greatest common divisor of $\nu_1$ and $\nu_2$, let $\gamma_i$ and $\delta$ be the integers $\nu_i/\nu$ and $m/(\gamma_1\gamma_2\nu)$, respectively. 
Then we have $\mu_1=\delta\gamma_2$, $\mu_2=\delta\gamma_1$ and $m=\delta\gamma_1\gamma_2\nu$. 
We fix a primitive $m$th root $\zeta_m$ of the unity, and we may assume that $\zeta_{\nu_i}=\zeta_m^{\mu_i}$.  
For each intersection $P\in\mcC_1\cap\mcC_2$, we also fix an open neighborhood $U_P$ of $P$ in $\PP^2$ such that $p_{n}^{-1}(U_P)\cong U_P\times\CC$. 
For the continuous maps $h_i:\mcC_i\to\mcT_{\mu_i n}$ ($i=1,2$) (resp. $F_\mcB:\mcC\to\mcT_n$), 
let $h_{i,P}:U_P\to\CC$ (resp. $F_{\mcB,P}:U_P\to\CC$) be the function such that $h_i(Q)=(Q,h_{i,P}(Q))\in\mcT_{\mu_i n}$ (resp. $F_\mcB(Q)=(Q,F_{\mcB,P}(Q))$) for any $Q\in U_P$ under a fixed isomorphism $p_n^{-1}(U_P)\cong U_P\times\CC$. 
For $P\in\mcC_1\cap\mcC_2$, let $d_P$ be a $\mu_1$th root of $h_{1,P}(P)$. 
Then we have 
\[ F_{\mcB,P}(P)=h_{1,P}^{\nu_1}(P)=h_{2,P}^{\nu_2}(P)=d_P^{m}, \ \ \mbox{and} \ \ h_{2,P}(P)=\zeta_{\nu_2}^{a_P}d_P^{\mu_2}   \]
for some integer $0\leq a_P<\nu_2$. 
Then we have the following claim. 

\begin{claim}\label{claim. c_1 and c_2}
The intersection $\widetilde{\mcC}_{1j_1}\cap\widetilde{\mcC}_{2j_2}\cap(\phi'_\mcB)^{-1}(P)$ is non-empty if and only if $j_1-j_2\equiv a_P \pmod{\nu}$. 
\end{claim}
\begin{proof}
The intersection $\widetilde{\mcC}_{1 j_1}\cap\widetilde{\mcC}_{2 j_2}\cap(\phi'_\mcB)^{-1}(P)$ is non-empty if and only if there exists a complex number $t_P\in\CC$ such that $t_P^{\mu_1}=\zeta_{\nu_1}^{j_1-1}d_P^{\mu_1}$ and $t_P^{\mu_2}=\zeta_{\nu_2}^{a_P+j_2-1}d_P^{\mu_2}$. 
The later condition is equivalent to 
\[ j_1-j_2+(\beta_1\gamma_1-\beta_2\gamma_2)\nu\equiv a_P \pmod{m}  \]
for some integers $\beta_1$ and $\beta_2$. 
Hence, if $\widetilde{\mcC}_{1j_1}\cap\widetilde{\mcC}_{2j_2}\cap(\phi'_\mcB)^{-1}(P)\ne\emptyset$, then $j_1-j_2\equiv a_P \pmod{\nu}$. 

Conversely, if $j_1-j_2\equiv a_P\pmod{\nu}$, then $a_P=j_1-j_2+b\nu$ for some integer $b$. 
Since $\gamma_1$ and $\gamma_2$ are coprime, there exist two integers $\alpha_1$ and $\alpha_2$ such that $\alpha_1\gamma_1-\alpha_2\gamma_2=1$; 
hence, by putting $\beta_i=b\alpha_i$ for $i=1,2$, $\beta_1\gamma_1-\beta_2\gamma_2=b$; 
this implies that $\widetilde{\mcC}_{1j_1}\cap\widetilde{\mcC}_{2j_2}\cap(\phi'_\mcB)^{-1}(P)\ne\emptyset$. 
\end{proof}

We next prove the following claim. 

\begin{claim}\label{claim. c_1j and c_1j'}
Let $\widetilde{\mcC}_{1 j}$ and $\widetilde{\mcC}_{1 j'}$ be closures of connected components $\widetilde{\mcC}_{1 j}'$ and $\widetilde{\mcC}_{1 j'}'$ of $(\phi'_\mcB)^{-1}(\mcC_1\setminus B)$, respectively. 
Then $\widetilde{\mcC}_{1 j}$ and $\widetilde{\mcC}_{1 j'}$ are contained in the closure $\widetilde{\mcC}$ of a connected component of $(\phi'_\mcB)^{-1}(\mcC\setminus B)$ if and only if 
\[ [j-j']_{\nu}\in\left\langle [a_P-a_Q]_{\nu} \mid P,Q\in\mcC_1\cap\mcC_2 \right\rangle\subset\ZZ/\nu\ZZ, \]
where $[r]_{\nu}$ is the image of the integer $r$ in $\ZZ/\nu\ZZ$. 
\end{claim}

\begin{proof}
Two connected components $\widetilde{\mcC}_{1j}', \widetilde{\mcC}_{1 j'}'$ of $(\phi'_\mcB)^{-1}(\mcC_1\setminus B)$ are contained in  a connected component $\widetilde{\mcC}'$ of $(\phi'_\mcB)^{-1}(\mcC\setminus B)$ if and only if  
there exists a path $p:[0,1]\to(\phi'_\mcB)^{-1}(\mcC\setminus B)$ such that $p(0)\in\widetilde{\mcC}'_{1 j}$ and $p(1)\in\widetilde{\mcC}'_{1 j'}$. 
We may assume that $\phi'_\mcB\circ p(0)$ and $\phi'_\mcB\circ p(1)$ are not intersection points of $\mcC_1$ and $\mcC_2$, and that $\phi'_\mcB\circ p(s)\in\mcC_1\cap\mcC_2$ for $0< s< 1$ if and only if  
either $\phi'_\mcB\circ p(s+\epsilon)\in\mcC_1$ and $\phi'_\mcB\circ p(s-\epsilon)\in\mcC_2$, or  $\phi'_\mcB\circ p(s+\epsilon)\in\mcC_2$ and $\phi'_\mcB\circ p(s-\epsilon)\in\mcC_1$ for any $0<\epsilon\ll 1$. 
Since $\phi'_\mcB\circ p(0)$ and $\phi'_\mcB\circ p(1)$ are two points of $\mcC_1$, the number of points $\Ima(\phi'_\mcB\circ p)\cap\mcC_1\cap\mcC_2$ is even. 
Let $0<s_1<\dots<s_{2k'}<1$ be the sequence such that $\Ima(\phi'_\mcB\circ p)\cap\mcC_1\cap\mcC_2=\{\phi'_\mcB\circ p(s_1),\dots,\phi'_\mcB\circ p(s_{2k'})\}$. 
Let $\widetilde{\mcC}_{1 j_i}'$ and $\widetilde{\mcC}_{2 l_i}'$ for $i=1,\dots, k'$ be the connected components of $(\phi'_\mcB)^{-1}(\mcC_1\setminus B)$ and $(\phi'_\mcB)^{-1}(\mcC_2\setminus B)$ such that $p(s)\in \widetilde{\mcC}'_{1 j_i}$ for $s_{2i}<s<s_{2i+1}$ and $p(s)\in\widetilde{\mcC}'_{2 l_i}$ for $s_{2i-1}<s<s_{2i}$, respectively, 
where $s_{2k'+1}=1$, hence $\widetilde{\mcC}'_{1j_{k'}}=\widetilde{\mcC}'_{1j'}$. 
Putting $P_i:=\phi'_\mcB\circ p(s_i)$ for $i=1,\dots,2k'$, by Claim~\ref{claim. c_1 and c_2}, we obtain 
\begin{align} 
j_{i-1}-l_{i}&\equiv a_{P_{2i-1}} \pmod{\nu},  \mbox{ and } \label{eq. 1} \\
j_{i}-l_i&\equiv a_{P_{2i}} \pmod{\nu}  \label{eq. 2}
\end{align}
for $i=1,\dots,k'$, where $j_0=j$. 
Hence, if $\widetilde{\mcC}'_{1j}$ and $\widetilde{\mcC}'_{1j'}$ are contained in a connected component $\widetilde{\mcC}'$ of $(\phi'_\mcB)^{-1}(\mcC\setminus B)$, then we have 
\[ j-j'\equiv \sum_{i=1}^{k'}(a_{P_{2i-1}}-a_{P_{2i}})\pmod{\nu}. \]

Conversely, suppose that $j-j'\equiv \sum_{i=1}^{k'}(a_{P_{2i-1}}-a_{P_{2i}})\pmod{\nu}$. 
We can find integers $0<j_i<\nu_1$ and $0<l_i<\nu_2$ for $i=1,\dots,k'$ satisfying the equations (\ref{eq. 1}) and (\ref{eq. 2}). 
By Claim~\ref{claim. c_1 and c_2}, there exists a path $p:[0,1]\to(\phi'_\mcB)^{-1}(\mcC\setminus B)$ with $p(0)\in\widetilde{\mcC}'_{1j}$ and $p(1)\in\widetilde{\mcC}'_{1j'}$. 
Therefore, $\widetilde{\mcC}'_{1j}$ and $\widetilde{\mcC}'_{1j'}$ are contained in a connected component $\widetilde{\mcC}'$ of $(\phi'_\mcB)^{-1}(\mcC\setminus B)$. 
\end{proof}

Fix a intersection point $P_0$ of $\mcC_1$ and $\mcC_2$, we may assume that $a_{P_0}=0$ after multiplying $\zeta_{\nu_2}^{-a_{P_0}}$ to $h_2$. 
Then we have 
\[ \langle [a_P-a_Q]_{\nu} \mid P,Q\in\mcC_1\cap\mcC_2 \rangle = \langle [a_P]_{\nu} \mid P\in \mcC_1\cap\mcC_2\rangle \cong \ZZ/\lambda\ZZ, \]
where $\lambda$ is the greatest common divisor of $\nu$ and $a_P$ for $P\in\mcC_1\cap\mcC_2$. 
We put $\lambda':=\nu/\lambda$. 
Then we obtain 
\begin{align*} 
h_{1,P}^{\gamma_1\lambda'}(P)&=d_P^{\mu_1\gamma_1\lambda'}=d_P^{\delta\gamma_1\gamma_2\lambda'}, \mbox{and} \\ 
h_{2,P}^{\gamma_2\lambda'}(P)&=(\zeta_{\nu_2}^{a_P}d_P^{\mu_2})^{\gamma_2\lambda'}=\zeta_{\nu_2}^{\gamma_2\nu (a_P/\lambda)}d_P^{\mu_2\gamma_2\lambda'}=d_P^{\delta\gamma_1\gamma_2\lambda'} 
\end{align*}
for any $P\in\mcC_1\cap\mcC_2$. 
We define the continuous map $h:\mcC\to\mcT_{\mu n}$ by $h(P)=h_i^{\gamma_i\lambda'}(P)$ if $P\in\mcC_i$ for $i=1,2$, where $\mu:=\gamma_1\mu_1\lambda'=\gamma_2\mu_2\lambda'=m/\lambda$. 
The map $h$ satisfies the conditions (ii-a), (ii-b) and (ii-c). 
Moreover, by Claim~\ref{claim. c_1j and c_1j'}, $t_n^{\mu}-\zeta_{\lambda}^{j-1}h=f_\mcC=0$ defines a subset $\widetilde{\mcC}_j$ of $(\phi')^{-1}(\mcC)$ such that $\widetilde{\mcC}_j\setminus(\phi')^{-1}(B)$ is connected for $j=1,\dots,\lambda$. 
\end{proof}

We next prove Corollary~\ref{cor. cn}.

\begin{proof}[Proof of Corollary~\ref{cor. cn}]
The existence of a divisor $D$ on $\PP^2$ with $\lambda D|_{\mcC}=\mcB|_{\mcC}$ implies the existence of a global section $H\in \HH^0(\PP^2,\mcO(\mu n))$ with $H^{\lambda}|_{\mcC}=F_\mcB|_{\mcC}$. 
Hence it is sufficient to prove that, for any continuous map $h:\mcC\to\mcT_{\mu n}$ satisfying the conditions (ii-a), (ii-b) and (ii-c) in Theorem~\ref{thm. cn}, there exists a global section $H\in\HH^0(\PP^2,\mcO(\mu n))$ such that $H|_{\mcC}=h$. 

Let $\mcC=\sum_{i=1}^{k}C_i$ be the irreducible decomposition of $\mcC$. 
We prove the above statement by induction on the number $k$ of irreducible components of $\mcC$. 
In the case $k=1$, it is obvious by the condition (ii-c).  
Suppose $k>1$, and put $\mcC_1:=\sum_{i=1}^{k-1}C_i$ and $\mcC_2:=C_k$. 
Let $h:\mcC\to\mcT_{\mu n}$ be a continuous map satisfying (ii-a), (ii-b) and (ii-c). 
By the assumption of induction, there exist two global sections $H_1, H_2\in\HH^0(\PP^2,\mcO(\mu n))$ such that $H_i|_{\mcC_i}=h|_{\mcC_i}$ for $i=1,2$. 
We have the following exact sequence 
\[ \mcO_{\PP^2}(\mu n-c_1)\otimes\mcO_{\PP^2}(\mu n-c_2)\overset{(f_1,-f_2)}{\longrightarrow}\mcO_{\PP^2}(\mu n)\to\mcO_{Z}(\mu n)\to 0, \]
where $c_i$ is the degree of $\mcC_i$, $f_i$ is a global section of $\mcO_{\PP^2}(c_i)$ defining $\mcC_i$ for $i=1,2$, and $Z$ is the scheme defined by the ideal sheaf generated by $f_1$ and $f_2$. 
Since $\mcC$ is a nodal curve, $\mcC_1$ and $\mcC_2$ intersect transversally. 
Hence $\mcO_Z(\mu n)$ is the skyscraper sheaf with support $\mcC_1\cap\mcC_2$ such that its stalk at each intersection $P\in\mcC_1\cap\mcC_2$ is isomorphic to $\CC$. 
Since $H_i|_{\mcC_i}=h|_{\mcC_i}$ for $i=1,2$, the image of $H_1-H_2$ on $\mcO_Z(\mu n)$ is zero. 
Hence there exist two global section $g_i\in\HH^0(\PP^2,\mcO(\mu n-c_i))$ for $i=1,2$ such that $f_1g_1-f_2g_2=H_1-H_2$. 
Then $H:=H_1-f_1g_1=H_2-f_2g_2$ satisfies $H|_{\mcC}=h$. 
\end{proof}

\section{Artal arrangements of degree $b$}

In \cite{artal}, Artal gave a Zariski pair $(\mcA_1,\mcA_2)$, where $\mcA_i$ is an arrangement of a smooth cubic $C_i$ and its three inflectional tangents $L_{i,j}$ ($j=1,2,3$) for each $i=1,2$, i.e., $C_i\cap L_{i,j}$ is just one point; 
putting $P_{i,j}$ as the tangent point of $C_i$ and $L_{i,j}$, the three point $P_{1,j}$ ($j=1,2,3$) are collinear, but the three points $P_{2,j}$ are not collinear. 
He distinguished the embedded topologies of $\mcA_1$ and $\mcA_2$ by their Alexander polynomials. 
It is possible to distinguish their embedded topologies by other invariants; the existence of dihedral coverings, the splitting number and the linking set (see \cite{tokunaga, ban-ben-shi-tok}). 
In this section, we generalize Artal's Zariski pair by using the connected number. 
First we define $k$-Artal arrangements of degree $b$. 

\begin{defin}\rm
Let $B$ be a smooth curve on $\PP^2$ of degree $b\geq 3$ having $k$ total inflection points $P_1,\dots,P_k$, and 
let  $L_i$ ($i=1,\dots,k$) be the tangent line of $B$ at $P_i$. 
We call $B+\sum_{i=1}^kL_i$ a \textit{$k$-Artal arrangement of degree $b$} if any three lines of $L_i$ ($i=1,\dots,k$) do not meet just one point. 
In the case of $k=3$, we call a $3$-Artal arrangement of degree $b$ an \textit{Artal arrangement of degree $b$}. 
\end{defin}

\begin{rem}
In \cite{ban-ben-shi-tok}, a $k$-Artal arrangement is defined as an arrangement of one smooth cubic and its $k$ inflectional tangents. 
A $k$-Artal arrangement in \cite{ban-ben-shi-tok} is our $k$-Artal arrangement of degree $3$. 
\end{rem}

We construct Zariski $k$-plets of Artal arrangements of degree $b\geq 4$ by using Fermat curves. 
For an integer $\mu\geq2$, let $F_{\mu}$ be the curve defined by $x^{\mu}+y^{\mu}+z^{\mu}=0$. 
For each $i=1,2,3$ and $j=1,\dots,\mu$, let $P^{\mu}_{i,j}$ be the following point
\[ P^{\mu}_{1,j}=(1{:}\zeta_{2\mu}^{2j-1}{:}0), \  P^{\mu}_{2,j}=(0{:}1{:}\zeta_{2\mu}^{2j-1}), \ P^{\mu}_{3,j}=(\zeta_{2\mu}^{2j-1}{:}0{:}1), \] where $\zeta_{2\mu}$ is a primitive $2\mu$-th root of unity. 
For $\mu\geq 3$, $F_{\mu}$ is the Fermat curve of degree $\mu$, and $P^{\mu}_{i,j}$ are the $3\mu$ total inflection points of $F_{\mu}$. 
Let $L^{\mu}_{i,j}$ be the tangent line of $F_{\mu}$ at $P^{\mu}_{i,j}$. 
Note that $L^{\mu}_{1,j}$, $L^{\mu}_{2,j}$ and $L^{\mu}_{3,j}$ are defined by 
\[ \zeta_{2\mu}^{2j-1}x-y=0, \ \zeta_{2\mu}^{2j-1}y-z=0 \ \mbox{and} \ \zeta_{2\mu}^{2j-1}z-x=0, \  \]
respectively, and that $L^{\mu}_{1,j_1}\cap L^{\mu}_{2,j_2}\cap L^{\mu}_{3,j_3}=\emptyset$ for any $j_1,j_2,j_3$. 
From Carnot's theorem (cf. \cite[Lemma~2.2]{coppens-kato}), we have the following lemma. 

\begin{lem}[Carnot, {cf. \cite[Lemma~2.2]{coppens-kato}}]\label{lem. carnot}
Fix integers $\mu, j_1,j_2,j_3$. 
There exists a divisor $D$ on $\PP^2$ of degree $d$ such that $D|_{L^{\mu}_{i,j_i}}=d P^{\mu}_{i,j_i}$ for all $i=1,2,3$ if and only if $(\zeta_{2\mu}^{2j_1+2j_2+2j_3-3})^{2d}=1$. 
\end{lem}

Then we obtain the following proposition by Corollary~\ref{cor. cn} and Lemma~\ref{lem. carnot}. 

\begin{prop}\label{prop. artal}
Let $h_{\mu}, f_{\mu,i}$ ($i=1,2,3$) be the following polynomials 
\begin{align*}
h_{\mu} &:= x^{\mu}+y^{\mu}+z^{\mu}, \\
f_{\mu,1} &:= \zeta_{2\mu}x-y, \\
f_{\mu,2} &:= \zeta_{2\mu}y-z, \\
f_{\mu,3} &:= \zeta_{2\mu}^{2\mu-1}z-x. 
\end{align*}
Let $\nu$ be a positive integer, and put $b:=\mu\nu$. 
Let $B_{b,\mu}$ be the curve of degree $b$ defined by 
\[ f_{\mu,1}f_{\mu,2}f_{\mu,3}\,g+h_{\mu}^{\nu}=0, \]
where $g$ is a general homogeneous polynomial of degree $b-3$, i.e., $B_{b,\mu}$ is smooth. 
Let $\phi_{B_{b,\mu}}:X_{b,\mu}\to\PP^2$ be the cyclic cover of degree $b$ branched at $B_{b,\mu}$, 
and put $L_{\mu}:=L^{\mu}_{1,1}+L^{\mu}_{2,1}+L^{\mu}_{3,\mu}$. 
Then $c_{\phi_{B_{b,\mu}}}(L_{\mu})=\nu$. 
\end{prop}

Since the connected number $c_{\phi_{B}}(L)$ is divisible by $b$ for an Artal arrangement of degree $b$, $B+L$, where $L=L_1+L_2+L_3$, Proposition~\ref{prop. artal} implies that, for any possible number $\nu$ as the connected number of Artal arrangements of degree $b$, i.e., any divisor $\nu$ of $b$,  there exists an Artal arrangement of degree $b$, $B+L$, with $c_{\phi_B}(L)=\nu$. 

\begin{thm}
Let $b$ be a positive integer, and let $\mu_1,\dots,\mu_k$ be distinct divisors of $b$. 
Assume that all $B_{b,\mu_i}$ are smooth. 
Let $\mcB_{b,i}$ be the Artal arrangement $B_{b,\mu_i}+L_{\mu_i}$ of degree $b$ constructed as above. 
Then $(\mcB_{b,1},\dots,\mcB_{b,k})$ is a Zariski $k$-plet.  
\end{thm}

Taketo Shirane

National Institute of Technology, Ube College, 2-14-1 Tokiwadai, Ube 755-8555, Yamaguchi Japan

\textit{E-mail address}: {\tt tshirane@ube-k.ac.jp}
\end{document}